\documentclass[12pt]{amsart}

\usepackage{mathrsfs}
\usepackage{bbm}
\usepackage{amsrefs}
\usepackage{enumerate}
\usepackage[hidelinks]{hyperref}

\usepackage[utf8]{inputenc}
\usepackage{pgf,tikz}
\usetikzlibrary{arrows}

\theoremstyle{definition}

\newtheorem{thm}{Theorem}
\newtheorem{lem}[thm]{Lemma}
\newtheorem{cor}[thm]{Corollary}
\newtheorem{defn}[thm]{Definition}
\newtheorem{conj}[thm]{Conjecture}
\newtheorem{prop}[thm]{Proposition}

\newtheorem{exa}[thm]{Example}

\newcommand{\R}{\mathbb R\relax}

\title[On Spaces of Infinitesimal Motions]{On Spaces of Infinitesimal Motions and Henneberg Extensions}
\author{James Cruickshank}
\address{School of Mathematics, Statistics and Applied Mathematics, National University of Ireland Galway, Ireland}
\email{james.cruickshank@nuigalway.ie}

\begin{document}

\maketitle

\begin{abstract}
    We investigate certain spaces of infinitesimal motions arising 
    naturally in the rigidity theory of bar and joint frameworks. 
    We prove some structure theorems for these spaces and as a 
    consequence are able to deduce some special cases of a long
    standing conjecture of Graver, Tay and Whiteley 
    concerning Henneberg extensions and generically rigid graphs.
\end{abstract}

\section{Introduction}
\label{sec:Motivation}

\subsection{Motivation}
In the rigidity theory of bar and joint frameworks the Henneberg
extensions play an important role. Let \( G=(V,E) \) be a simple 
finite graph and let \( n \) and \( k \) be positive integers,
and suppose that \( X \subset V \) with 
\( |X| = n+k \) and \( F \subset E(X) \) with \( |F| =k\). 
We can form a new graph \( G' \) by deleting all the edges in \( F \)
and adding a new vertex of degree 
\( n+k \) that is adjacent to all the vertices of \( X \). We say that 
\( G' \) is an {\em \( n \) dimensional Henneberg \( k \)-extension}
of \( G \) that 
is supported on the vertex set \( X \) and the edge set \( F \).
A major stumbling block to a full 
understanding of the generic rigidity theory of three dimensional bar and
joint frameworks is the lack of good sufficient 
conditions that ensure that a three dimensional Henneberg 2-extension
preserves generic rigidity.
By contrast the planar situation is relatively 
well understood, and we have an extensive theory of planar rigidity
stemming largely from Laman's Theorem (\cite{MR0269535}). 
Laman's result can be proved by an inductive argument 
based on two dimensional Henneberg extensions, so it is obviously of interest
to see if similar lines of reasoning can be pursued in the three
dimensional case.

Given a graph \( G \) and vertices \( i \) and \( j \) of \( G \), 
we say that \( ij \) is an {\em implied edge} 
in three dimensions if any flex of a generic
three dimensional 
framework based on \( G \) is an infinitesimal isometry of the 
set of points corresponding to \( \{i,j\} \) (see section 
\ref{sec:PointsAdMotions} for the relevant background material). 
In particular, 
every edge of \( G \) is an implied edge, but there may be implied edges
that are not edges of \( G \). An {\em implied \( K_4 \)} is a set 
of implied edges of \( G \) that form a complete graph on 4 vertices.
The following Conjecture of Graver, Tay and Whiteley remains open.
For a discussion of this conjecture and many other related problems
the reader should consult \cite{MR1251062} (wherein it is referred to as the 
Henneberg Conjecture) or, for a more recent discussion,
\cite{MR2269396} (Conjecture 2.7).

\begin{conj}
    \label{conj:GraverTayWhiteleyConjecture}
    Let \( G=(V,E) \) be a generically 3-isostatic graph and suppose 
    that \( X \) is a subset of \( V \) such that \( |X| =5 \).
    Let \( e \) and \( f \) be distinct edges whose vertices
    are in \( X \). Suppose that \( G-\{e,f\} \) does not
    contain any implied \( K_4 \) whose vertices are a subset
    of \( X \).  Then the 
    Henneberg extension that removes \( e \) and \( f \) and 
    adjoins a vertex with neighbour set \( X \) results in 
    a generically 3-rigid graph. 
\end{conj}

A proof of this conjecture would be a significant step towards a 
better understanding of generic rigidity in 3-space as it allow us 
to characterise generically rigid graphs as precisely those that
may be obtained from \( K_3 \) by a sequence of allowable Henneberg 
\(k\)-extensions, where \( k=0,1 \) or \( 2 \).

Conjecture \ref{conj:GraverTayWhiteleyConjecture}
has been resolved in certain special cases. For example 
Graver (\cite{MR1251062}, Section 5.5) 
has shown that if \( E(X) -\{e,f\} \) contains
a 5-cycle then the conjecture is true. Jackson and Owen 
(\cite{JacksonOwenNotes})
have shown that if \( E(X) \) contains a 3-cycle containing \( e \),
and \( e \) and \( f \) are disjoint then the conjecture is true. 

In this article we investigate the infinitesimal dynamic properties 
of rigid frameworks that admit non-rigid Henneberg 2-extensions
More precisely, we should say that the underlying graph admits
a 2-extension that is not generically rigid.  
As an application of our main result, we are able to prove 
Conjecture \ref{conj:GraverTayWhiteleyConjecture} in the case 
where \( |E(X)|\geq 7 \) - generalising the 
result of Graver mentioned above. Note that, in the case where 
\( e \) and \( f \) are vertex disjoint, this result is already implied 
by the result of Jackson and Owen mentioned above. However, 
their proof does rely on the vertex disjointness of \( e \) and \( f \).
Thus, our result is logically independent of theirs.

\subsection{Structure of the paper}
In Section \ref{sec:PointsAdMotions} we review some of the basic 
rigidity theory of frameworks - there are no new results in this section. 
However we present this material in a matrix algebraic setting 
that is appropriate to our later needs.

Sections \ref{sec:PAdmissible}, \ref{sec:KNoneFrameworks} and 
\ref{sec:FivePoints} are devoted to the study of \( p \)-admissibility.
This is the central concept of the paper and essentially captures
the infinitesimal dynamic properties of any rigid framework that admits
a Henneberg extension that is not generically rigid.
We make use of a standard matrix identity known as the Sherman-Morrison
formula to organise some of the messy polynomial equations that 
arise when considering the infinitesimal flexes of a framework. 
The main technical result of the paper is Theorem \ref{thm:MainTecchnicalResult}
which provides interesting geometric information about any possible 
counterexample to Conjecture \ref{conj:GraverTayWhiteleyConjecture}. 
It is also, we believe, of independent interest as it seems to suggest
a non-obvious connection between Euclidean rigidity and affine 
rigidity in the sense of \cite{ThurstonGortler}. 

Finally in Section \ref{sec:Applications}, we apply the results of the 
previous sections to prove some special cases of Conjecture \ref{conj:GraverTayWhiteleyConjecture}.
We also point out a fundamental obstruction to a proof of the general 
case.

\subsection{Notation}

\subsubsection{Matrices}
We write \( \R^n \) for the space of \( n \times 1 \) column vectors
with real entries.
Throughout this article, we will write \( \R^{m\times n} \) for the
set of matrices with real entries that have \( m \) rows and \( n \)
columns. 
For \( p \in \R^{m \times n} \), we write \( p_i \) for the \( i \)th 
column of \( p \).
We write \( \mathbbm 1 \) for the column vector
all of whose entries are \( 1 \) and we write \( I \) for the  identity
matrix.
We write \( m^T \) for the transpose of the matrix \( m \).
In particular, for \( u,p \in \R^n \) the dot product of \( u \) 
and \( p \) is \( u^Tp \).

\subsubsection{Graphs}
In the article graphs are simple undirected graph with no loops
or parallel edges. If \( G=(V,E) \) is 
such a graph and \( E'\subset E \), then 
\( G-E' \) denotes the graph obtained by removing the 
all the edges in  \( E' \)
from \( G \). If \( E' \subset E \) then \( V(E) \) is the set 
of vertices in \( V \) spanned by \( E' \). If \( V'\subset V \) then 
\( E(V') \) is the set of edges in \( E \) spanned by \( V' \). 
\( K(X) \) denotes the complete graph with vertex set  \( X \), 
\( K_n \) denotes the complete graph with \( n \) vertices and
\( K_{p,q} \) denotes the complete bipartite graph with vertex 
parts of order \( p \) and \( q \).

\subsubsection{Miscellaneous}
We write \( \{\{ \cdots\}\} \) to indicate a multiset.
Statements such as ``\( P(x) \) is true for {\em almost all}
\( x \) in \( \R^3 \)'' mean that there is an open dense 
subset \( O \subset \R^3 \) such that \( P(x) \) is true
for all \( x \in O \).

\section{Point Configurations and Motions}
\label{sec:PointsAdMotions}

\subsection{Infinitesimal Isometries of Euclidean Space}

Given a vector field \( \xi :\R^{n}\to \R^{n} \), let \( \xi_x \) 
denote the value of the vector field at a point 
\( x \in \mathbb R^{n} \). 
We say that \( \xi \) is an infinitesimal isometry of 
\( \mathbb R^{n} \)
if 
\[ (\xi_x-\xi_y)^T(x-y)=0  \]
for all \( x, y \in \mathbb R^{ n} \). 
The set of infinitesimal isometries
form a linear subspace of the space of vector fields. Let \( \mathcal I \)
denote this subspace. Let 
\( \mathcal I_0 = \{ \xi \in \mathcal I: \xi_0 = 0\} \).

\begin{lem}
    \label{lem:SkewSymmetrics}
    Given an \( n\times n \)
    skew symmetric matrix \( a \), let \( \xi^a \)
    be the vector field on \( \R^{n} \) defined by 
    \( \xi^a_x = ax \).
    The mapping \( a \mapsto \xi^a\) is a linear
    isomorphism between the space of  skew symmetric matrices 
    and \( \mathcal I_0 \).
\end{lem}

\begin{proof}
    Let \( \xi \in \mathcal I_0 \). Let \( a_\xi \) be the \( n \times n \)
    matrix whose \( i \)th column is \( \xi_{e_j} \) where \( e_j \)
    is the \( j \)th standard basis vector of \( \R^{n} \). 
    We leave it to the reader to verify that \( \xi\mapsto a_\xi \)
    is the inverse of the map defined in the statement.
\end{proof}

\begin{lem} 
    \label{lem:SkewSymmetricsAndTranslations}
    For any \( \xi \in \mathcal I  \) there are unique \( t \in 
    \R^{n } \) and skew symmetric 
    \( a \in \R^{n \times n} \) such that
    \( \xi_x = t +ax \) for all \( x \in \R^{
    n}\).
\end{lem}

\begin{proof}
    Let \( t = \xi_0 \) and observe that \( t \) is the 
    unique element of \( \R^{n} \) such that 
    \( \xi-t \in \mathcal I_0 \).
    Now the lemma follows immediately from Lemma 
    \ref{lem:SkewSymmetrics}.
\end{proof}

\subsection{Point Configurations}
\label{subsec:PointConfigurations}

Let \( n \) be a fixed positive integer. 
A \( k \)-point configuration,
\( p \), in \( \mathbb R^n
\) is just an element of \( \R^{n\times k} \). 
We think of \( p_i \) as the \( i \)th point of the configuration. 

Given a point configuration \( p \in \R^{n \times k} \), 
An {\em infinitesimal motion} of \( p \) 
is a matrix, \( u\in \R^{n \times k} \).
We think of \( u_i \) as the velocity 
vector associated to \( p_i \).
An {\em infinitesimal isometry} of \( p \) is a motion 
\( u \in \R^{n \times k} \) that satisfies 
\begin{equation}
    (u_i-u_j)^T(p_i-p_j)=0
    \label{eq:InfinitesimalIsometryCondition}
\end{equation}
for all \( i,j \). We will also refer to an isometry 
of \( p \) as a trivial motion of \( p \). Thus, it is 
important to note the distinction between nontrivial and
non zero in this context.

Observe that any \( \xi \in \mathcal I \) induces an infinitesimal
isometry \( u \), of \( p \), defined by \( u_i = \xi_{p_i} \). In this
way, we can linearly embed \( \mathcal I  \) in \( \R^{n \times k} \). 
We denote the image of \( \mathcal I \) in \( \R^{n \times k} \) by \( \mathcal I_p \).
It is easily seen that this embedding is injective if and only if 
the affine span of \( p \) has dimension at least \( n-1 \).
On the other hand, it is also easily seen that any 
infinitesimal isometry of \( p \)
is induced by a unique (global) 
infinitesimal isometry of the affine span of \( p \).
In particular, any infinitesimal isometry of \( p \)
is induced by some infinitesimal isometry of \( \R^{n} \).

\begin{lem}
    Given \( p \in \R^{n \times k} \) and \( u \in \R^{n \times k} \),
    then \( u \) is an infinitesimal isometry of \( p \) if and 
    only if there is some skew symmetric matrix \( a\in 
    \R^{n\times n}\) and some \( t \in \R^{n} \) such that
    \( u = t\mathbbm1^T + ap \).
\end{lem}

\begin{proof}
    As we have just observed, \( u \) is an infinitesimal 
    isometry of \( p \) if and only there is some 
    infinitesimal isometry \( \xi \) of \( \R^{n} \)
    such that \( \xi_{p_i} = u_i \). Now the result follows 
    immediately from Lemma \ref{lem:SkewSymmetricsAndTranslations}.
\end{proof}

We say that motions \( u \) and \( u' \) of \( p \) are \( p \)-equivalent
if \( u-u' \in \mathcal I_p \). We also can extend this notion to 
subspaces of \( \R^{n\times k} \). Thus, linear 
subspaces \( S \) and \( S' \) of \( \R^{n\times k} \) are said to 
to be {\em \( p \)-equivalent} if they have the same dimension and they
have the same image under the projections \( \R^{n\times k}
\to \R^{n\times k}/\mathcal I_P\). One readily checks that \( S \) 
and \( S' \) are \( p \)-equivalent if and only if there is 
a linear isomorphism \( f:S \to S' \) such that \( u-f(u) \in 
\mathcal I_p\) for all \( u \in S \).

The following notion will be useful later on. We say that a vector field
\( \xi \) on \( \mathbb R^n \) is a {\em linear field} if the function 
\( x \mapsto \xi_x \) is  linear. We say that \( \xi \) is an {\em affine
field} if the function \( x \mapsto \xi_x \) is an affine linear function.
Given a motion \( u \)
of a point configuration \( p \), we say that \( u \) is a {\em
linear motion}, respectively an {\em affine motion}, 
if it is the restriction to \( p \) of a linear vector field,
respectively an affine vector field.
We remark that if \( p \in \R^{n \times (n+1)} \) is in general position
then any \( u\in \R^{n \times (n+1)} \) is an affine motion of \( p \).
On the other hand the affine motions of a general 
\( q \in \R^{n\times (n+2)} \) form a codimension \( n \) affine
subspace of the vector space of all motions of \( q \).

\subsection{Frameworks}

In this section we review some elementary facts about frameworks.
A \( n \) dimensional {\em framework} is a pair 
\( \left( G,p \right) \) where \( G \) is 
a simple undirected graph with vertex set \( V \) and edge set
\( E \) and where \( p\in \R^{n \times |V|} \). 
A \( G \)-framework is a framework whose underlying graph
is \( G \).
For convenience, we will assume for the moment 
that \( V = \{1,\cdots,k\} \). 

An {\em infinitesimal flex} of \( (G,p) \) is 
a motion, \( u\in \R^{n \times |V|} \), that satisfies
\begin{equation}
    (u_i-u_j)^T(p_i-p_j)=0
    \label{eq:FlexDefinition}
\end{equation}
for all \( ij \in E \).
We say that an infinitesimal flex is trivial if it belongs to 
\( \mathcal I_p \). A framework is {\em infinitesimally
rigid} if every infinitesimal flex of the framework is trivial.
A framework is {\em isostatic} if it is infinitesimally rigid and if
the deletion of any edge from \( G \) results in a framework
that is not infnitesimally rigid.

We will be 
particularly concerned with frameworks whose underlying point 
configuration is generic.
We say that \( p \) is {\em generic} if the multiset of its entries is 
algebraically independent over \( \mathbb Q \). 
We observe that  generic
configurations are in particular in general position. Also, we observe 
that if \( p \) is generic then every submatrix of \( p \) has
maximal rank.

We say that a graph \( G \) is {\em generically \( n \)-rigid}
if \( (G,p) \) is rigid for any generic configuration 
\( p \) in \( \mathbb R^{ n\times |V|} \).  
It can be shown that if \( (G,p) \) is infinitesimally 
rigid for some (possibly non-generic) \( p \in \R^{n \times |V|} \)
then \( G \) is generically rigid. On the other hand,
it can happen that for a generically rigid graph \( G \)
there are certain (non-generic) \( G \)-frameworks that are not 
infinitesimally rigid. 
As mentioned in Section \ref{sec:Motivation}, a 
fundamental open problem in combinatorial rigidity
theory is to find a good characterisation of generically 
\( n \)-rigid graphs for \( n \geq 3 \). 
A dimension counting argument suggests that \( n \) dimensional 
Henneberg extensions might preserve generic \( n \)-rigidity. However, as
is well known, this is not in general true
- see Example \ref{exa:AdmissibleExampleOne} below. 

\section{\( p \)-admissible subspaces}
\label{sec:PAdmissible}

Let \( k \) and \( n \)  be positive integers such that \( k\geq n+1 \).
Suppose that \( K_{k,1} \) is the complete bipartite 
graph with vertex sets \( \{1,\cdots,k\} \) and \( \{k+1\} \).
For \( p \in \R^{n \times k}\) 
and \( x \in \R^{n} \), let 
\( p^x  = \begin{pmatrix} p &  x\end{pmatrix}\).

    \begin{defn}\label{defn:Admissibility}
A linear subspace \( S \leq \R^{n \times k} \) will be called {\em 
\( p \)-admissible}
if it has the following properties:
\begin{enumerate}
    \item \( S\cap \mathcal I_p = 0 \).
    \item 
        For almost all \( x \in \R^n \) there is some non-zero 
        \( u \in S \) such that \( u  \) is the restriction 
        to \( p \) of some motion 
        of the framework \( (K_{k,1},p^x) \).
\end{enumerate}
\end{defn}

    One may wonder why we restrict to an open dense subset 
    of \( \R^n \) in the definition of \( p \)-admissibility.
    We do this in order to avoid the complications that
    would arise from degenerate configurations 
    if we had to check the condition for all 
    \( x \in \R^n \). In all applications of this concept 
    that we have in mind, it is sufficient that condition
    (2) above be satisifed on an open dense subset.

In order to motivate Definition \ref{defn:Admissibility}
we consider the following situation. 
Let \( (G,\rho) \) be an isostatic framework in \( \R^n \)
and let \( p \in \R^{n\times k} \) be the first \( k \)
columns of \( \rho \).
We assume that \( V(G) = \{1,\cdots,l\} \) for some \( l\geq k \)
Let \( E' \subset E \) be a set of edges of \( G \) such that
\( k = |E'|+n \) and such that \( V(E') \subset \{1,\cdots,k\} \). 
Since \( (G,
\rho)\) is isostatic the framework \( (G-E',\rho) \) has 
a \( |E'| \) dimensional 
space of flexes \( \overline S \) such that any nontrivial
element of \( \overline S \) restricts to a nontrivial motion of 
\( p\). Thus \( \overline S \) induces a \( |E'| \) dimensional subspace 
\( S \) of \( \R^{n\times k} \) such that \( S\cap \mathcal I_p = 0 \).
It is clear that \( S \) is \( p \)-admissible if and only if 
almost every Henneberg \( |E'| \)-extension 
(of the framework) that deletes
\( E' \) and adjoins a vertex adjacent to \( \{1,\cdots,k\} \) results 
in a non-rigid framework. 
Thus, we see the relevance of understanding \( p \)-admissible subspaces
of \( \R^{3 \times 5} \). In particular, any rigid generic framework 
that admits a \( 2 \)-extension that is not generically rigid 
will give rise to a two dimensional \( p \)-admissible subspace of 
\( \R^{3\times 5}\) where \( p \) is some generic  element of \(
\R^{3 \times 5}\).

\begin{exa}
    \label{exa:AdmissibleExampleOne}
Let \( G \) be the graph obtained by removing a single 
edge from \( K_5 \) and let \( p \) be a generic embedding
of the vertex set into \( \R^3 \). It is well known (\cite{MR804977}
for example) and elementary
that a Henneberg 2-extension of \( (G,p) \) that removes two 
edges that are both incident to the same vertex of degree 3
results in a non-rigid framework. Up to a possible 
permutation of the vertices, the corresponding \( p \)-admissible
space is \( p \)-equivalent, as defined 
in Section \ref{subsec:PointConfigurations}, to the subspace 
\( \R^{3\times 5} \) generated by \( e_{11} \) and \( e_{21} \) (where
\( e_{ij} \) is the matrix with a \( 1 \) in position \( (i,j) \) and
zeroes everywhere else). 
\end{exa}

Indeed, Conjecture \ref{conj:GraverTayWhiteleyConjecture} 
is equivalent to the assertion that, for \( p \in \R^{3 \times 5} \),
the only \( p \)-admissible
subspaces that arise from the deletion of two edges from a generic 
isostatic framework whose vertices include \( p \) are those that are 
\( p \)-equivalent to the one described in
Example \ref{exa:AdmissibleExampleOne}. 
There are however, examples of \( p \)-admissible
spaces that are not equivalent to the one described in Example
\ref{exa:AdmissibleExampleOne}.

\begin{exa}
    \label{exa:AdmissibleExampleTwo}
    Let \( p \in \R^{3 \times 5} \) be in general position.
    Fix some constant \( k \in \R \) and let
    \begin{equation}
        \label{eq:AdmissibleExampleTwo}
        S= \{ u \in \R^{3\times 5}: u_1^T(p_1-p_2) = 0, u_2 = ku_1, u_3=u_4=u_5 = 0\}. \end{equation}
    Then  \( S \) is a two dimensional \( p \)-admissible space.
    One can see this by observing that the structure
    shown in Figure \ref{fig:NonGenericFramework} is not rigid 
    regardless of the position of the point \( x \). Note that in this
    structure the point \( p_2 \) 
    is a fixed point on the bar joining \( p_1 \) and \( y \).
    Thus, the structure is not an example of a 
    generic bar and joint framework.  However, in the context of this 
    example we are only interested in the infinitesimal motion 
    of \( p \) induced by a flex of the structure.
    It is clear the space of infinitesimal flexes of the 
    structure consisting of only the black bars 
    induces a space of motions that is  \( p \)-equvalent
    to the one 
    described by Equation 
    \ref{eq:AdmissibleExampleTwo}.
\end{exa}

\begin{figure}
\definecolor{qqqqff}{rgb}{0,0,1}
\begin{tikzpicture}[line cap=round,line join=round,>=triangle 45,x=1.0cm,y=1.0cm]
\clip(-4.3,-1.1) rectangle (3,6.3);
\draw (-2.48,2.76)-- (-1.38,1.8);
\draw (-1.38,1.8)-- (-0.7,2.48);
\draw (-0.7,2.48)-- (-2.48,2.76);
\draw [red](-2.48,2.76)-- (-0.08,5.08);
\draw [red](0.47,0.98)-- (-0.08,5.08);
\draw [red](-0.08,5.08)-- (-0.7,2.48);
\draw [red](-0.08,5.08)-- (-1.38,1.8);
\draw (-1.38,1.8)-- (-0.68,-0.5);
\draw (-0.7,2.48)-- (-0.68,-0.5);
\draw (-2.48,2.76)-- (-0.68,-0.5);
\draw (-0.68,-0.5)-- (1.24,1.96);
\draw [red](-0.08,5.08)-- (1.24,1.96);
\begin{scriptsize}
\fill [color=qqqqff] (-2.48,2.76) circle (1.5pt);
\draw[color=qqqqff] (-2.74,2.74) node {$p_3$};
\fill [color=qqqqff] (-1.38,1.8) circle (1.5pt);
\draw[color=qqqqff] (-1.56,1.7) node {$p_4$};
\fill [color=qqqqff] (-0.7,2.48) circle (1.5pt);
\draw[color=qqqqff] (-0.44,2.74) node {$p_5$};
\fill [color=qqqqff] (-0.08,5.08) circle (1.5pt);
\draw[color=qqqqff] (0.08,5.34) node {$x$};
\fill [color=qqqqff] (-0.68,-0.5) circle (1.5pt);
\draw[color=qqqqff] (-0.42,-0.49) node {$y$};
\fill [color=qqqqff] (1.24,1.96) circle (1.5pt);
\draw[color=qqqqff] (1.38,2.22) node {$p_1$};
\fill [color=qqqqff] (0.47,0.98) circle (1.5pt);
\draw[color=qqqqff] (0.75,1.0) node {$p_2$};
\end{scriptsize}
\end{tikzpicture}
\begin{caption}{}
\label{fig:NonGenericFramework}
\end{caption}
\end{figure}

One consequence of Example \ref{exa:AdmissibleExampleTwo} is that 
any proof of Conjecture \ref{conj:GraverTayWhiteleyConjecture} must
implicity or explicitly demonstrate that this \( p \)-admissible space
cannot arise by deleting two edges from a generic isostatic framework.
In fact, the situation is even more complicated as 
we will show below that there a many other essentially
inequivalent examples of \( p \)-admissible spaces for generic 
\( p \in \R^{3\times 5}  \). 

\section{\( K_{n,1} \)-frameworks}
\label{sec:KNoneFrameworks}

Observe that \( K_{5,1} \) is the union of two copies 
of \( K_{3,1} \) that have one edge in common. Thus 
in order to understand flexes of \( K_{5,1} \)-frameworks
we should first understand the flexes of 
\( K_{3,1} \)-framework and then consider flexes of 
two \( K_{3,1} \)-frameworks with a common edge.
In this section we will begin that programme by  
deriving some basic results concerning 
\( K_{n,1} \)-frameworks.

Let \( q \in \R^{n\times n}\) be a configuration of \( n \)
points in \( \R^n \) and let \( v \in \R^{n\times n} \)
be an infinitesimal motion of \( q \). 
For a square matrix \( b \), let \( \triangle(b) \) denote the column 
vector whose entries are the diagonal entries of \( b \).
For \( x\in \R^n \) such that 
\( \mathbbm1 x^T - q^T \) is invertible, define 
\[ \mathcal P(x,q,v) = (\mathbbm 1 x^T - q^T)^{-1} (v^Tx - \triangle(v^Tq))\]

We observe that, if \( q \) is an invertible matrix, then 
\( \mathbbm1x^T -q^T \) is invertible if and only if \( x \) 
does not belong to the affine span of the columns of \( q \). Moreover, 
in this case 
\begin{equation}\label{eq:ShermanMorrison} (\mathbbm 1x^T - q^T)^{-1} = -(q^T)^{-1}\left(I + \frac{\mathbbm 1(q^{-1}x)^T}{1- (q^{-1}x)^T\mathbbm 1}\right)\end{equation}
    
Identity (\ref{eq:ShermanMorrison}) is sometimes referred to as the 
Sherman-Morrison formula (see \cite{MR0040068} or \cite{MR0035118}) 
and is elementary to verify. 
The relevance of \( \mathcal P \) to our discussion is apparent from
the following observation.
(Recall that, for \( x \in \R^n \), \( q^x = \begin{pmatrix} q & x \end{pmatrix} \)).
\begin{lem}
    Suppose that \( q \) is an invertible matrix 
    and that \( x \) does not belong to the affine span 
    of the rows of \( q \). Then \( \mathcal P(x,q,v) \) 
    is the unique vector such that \( ( v, 
    \mathcal P(x,q,v) ) \) is a flex of the framework 
    \( (K_{n,1}, q^x) \). 
\end{lem}

\begin{proof}
    An elementary calculation verifies that 
    \[(\mathcal P(x,q,v) - v_i)^T(x-q_i) = 0\]
    for \( i=1,2,\cdots,n \). The uniqueness  statement follows
    from the fact that our hypotheses ensure that
    the vectors  \( x-q_i \), \( i=1,2,\cdots, n \)  are linearly independent.
\end{proof}

Thus, for almost all \( x \) and \( q \), we have 
\begin{equation}
    \mathcal P(x,q,v) =
    -(q^T)^{-1}\left(I + \frac{\mathbbm 1 (q^{-1}x)^T}{1-(q^{-1}x)^T\mathbbm1}
    \right) (v^Tx - \triangle (v^Tq))
    \label{eq:PFormula}
\end{equation}

In order to understand the flexes of a \( K_{5,1} \)-framework, we 
will have to consider solutions of an equation of the form
\[ \mathcal P(x,q,v) = \mathcal P(x,r,w) \] where 
\( q \) and \( r \) are three point configurations with a common 
point and \( v \) and \( w \) are motions that agree on the 
common vertex. This leads to rather complicated equations. So
we will also consider a certain limit
associated to the function \( \mathcal P \).
This argument is somewhat inspired by the viewpoint of coarse
geometry - that is that certain geometric problems can 
be simplified by looking at them from `far away'. 
Define  
\[ \mathcal L (x,q,v) = \lim_{t\rightarrow +\infty} \frac{\mathcal P(tx,q,v) }{t}.\]
\begin{lem}
    Suppose that \( q \) is invertible and that \( (q^{-1}x)^T\mathbbm1
    \neq 0\). Then \( \mathcal L(x,q,v) \) is defined and
    \begin{equation}
        \mathcal L(x,q,v) = -(q^T)^{-1}\left(I - 
        \frac{\mathbbm1 (q^{-1}x)^T}{(q^{-1}x)^T\mathbbm1}\right)
        v^Tx
        \label{eq:LFormula}
    \end{equation}
\end{lem}

\begin{proof}
    This follows easily by combining Equation (\ref{eq:PFormula})
    with the definition of \( \mathcal L(x,q,v) \).
\end{proof}

Note that \( (q^{-1}x)^T\mathbbm1 = 0  \) if and only if \( x  \)
belongs to the linear subspace that is parallel to the affine
span of the columns of \( q \).

While Equation (\ref{eq:LFormula}) may seem rather complicated still,
it is worth noting that the matrix \(  I - 
\frac{\mathbbm1 (q^{-1}x)^T}{(q^{-1}x)^T\mathbbm1}\) is just the matrix of the 
projection  parallel to \( \mathbbm1 \) onto the orthogonal 
complement of \( (q^{-1}x)^T \). Thus
the geometry of \( \mathcal L \) is more transparent than that
of \( \mathcal P \).

We also note that, given \( x \), 
\( \mathcal L(x,q,v) \in x^\perp \) for all 
\( q \) and \( v \). Thus, if \( \pi \)
is a fixed linear mapping on \( \R^n \) of rank \( n-1 \), then
for almost all \( x \), \( \mathcal L
(x,q,v)\) is determined by \( x \) and 
\( \pi(\mathcal L(x,q,v)) \). 
By contrast, it is not possible to 
infer that \( \mathcal P(x,q,v) \) belongs to any specific two
dimensional subspace of \( \R^{3} \) that depends only 
on \( x \).

\section{\( p \)-admissibility for 5 points in \( \R^3 \)}
\label{sec:FivePoints}

In this section we will use the results of the previous section 
to analyse  two dimensional \( p \)-admissible spaces where \( p \)
is a five point configuration in \( \mathbb R^3 \). Given that 
we are restricting our attention to such a special situation we
find it appropriate to introduce the following special notational
conventions.
For \( p \in \R^{3 \times 5} \), \( q \) is the matrix obtained by deleting
the second and third columns of \( p \), 
whereas \( r \) is the matrix obtained by 
deleting the fourth and fifth columns  of \( p \).
Similarly for \( u \in \R^{3 \times 5} \), \( v \) is the matrix obtained by 
deleting the second and third columns of \( u \), 
whereas \( w \) is the matrix
obtained by deleting the fourth and fifth columns of \( u \).

\begin{lem}
    \label{lem:LinearMotionLemmaOne}
    Let  \( p,u \in \R^{3 \times 5} \).
    Then \( u \) is a linear motion of \( p \) if and only 
    if there is some \( m \in \R^{3\times 3} \) such that
    \( v = mq \) and \( w = mr \). In particular,
    if \( q \) is invertible then \( u \) is a linear
    motion of \( p \) if and only if \( w = vq^{-1}r \).
\end{lem}

\begin{proof}
    By definition \( u \) is a linear motion of \( p \) 
    if and only if there 
    is some \( m \in \R^{3 \times 3} \) such that 
    \( u_i = mp_i \). 
\end{proof}

\begin{lem}
    \label{lem:pEquivalenceLemma}
    Suppose that \( u,u'\in \R^{3 \times 5} \) are motions
    of \( p \in \R^{3 \times 5} \). Then \( u \) and \( u' \)
    are \( p \)-equivalent if and only if there is some 
    skew symmetric matrix \( a\in \R^{3 \times 3} \)
    and some \( t \in \R^3 \) such that 
    \( v + aq+t\mathbbm1^T = v' \) and \( w+ ar+t\mathbbm1^T = w' \)
\end{lem}

\begin{proof}
    This follows from the definition of \( p \)-equivalence and
    Lemma \ref{lem:SkewSymmetricsAndTranslations}.
\end{proof}


Now we come to the main technical results of the paper.
For the rest of this section we assume that both 
\( q \) and \( r \) are invertible.
First we recharacterise \( p \)-admissibility of a subspace \( S \)
of \( \R^{3\times 5} \).
\begin{prop}
    \label{prop:PadmissibilityPropOne}
    Let \( S \) be a two dimensional subspace of \( \R^{3 \times 5} \)
    such that \( S\cap \mathcal I_p = 0 \). Then \( S \) is 
    \( p \)-admissible if and only if the linear mapping 
    \( h_x: S \to \R^{3} \) defined by 
    \( h_x(u) = \mathcal P(x,q,v) - P(x,r,w) \) has rank 
    at most one for almost all \( x \in \R^3\).
\end{prop}

\begin{proof}
    Observe that the linear mapping \( h_x \) defined in the 
    statement has rank at most one if and only if there 
    is some non-zero \( u \in S \) such that 
    \begin{equation}
        \label{eq:PEqualityCondition}
        \mathcal P(x,q,v) = \mathcal P(x,r,w) 
    \end{equation}
    (using the notational
    convention specified above). But Equation 
    (\ref{eq:PEqualityCondition}) is true if and only if 
    \( \begin{pmatrix} u & \mathcal P(x,q,v) \end{pmatrix} \)
        is a flex of \( (K_{5,1},p^x) \).
\end{proof}

Proposition \ref{prop:PadmissibilityPropOne} 
allows us to give a sufficient condition
for \( p \)-admissibility. 

\begin{prop}
    \label{prop:AdmissibilitySufficientCondition}
    Suppose that \( S \) is a two dimensional 
    subspace of \( \R^{3\times 5} \) with the following three properties.
    \begin{enumerate}
        \item[A1.] \( S\cap \mathcal I_p = 0 \)
        \item[A2.] Every \( u\in S \) is a linear motion of \( p \).
        \item[A3.] For all \( u \in S \) we have \( (q^T)^{-1}\triangle(v^Tq) =
                (r^T)^{-1}\triangle(w^Tr)\)
    \end{enumerate}
    Then \( S \) is \( p \)-admissible. 
\end{prop}

\begin{proof}
    By Lemma \ref{lem:LinearMotionLemmaOne},
    \( vq^{-1} = wr^{-1} \) for all \( u \in S \).
    Thus, combining Equation (\ref{eq:PFormula})
    with the definition of \( h_x \), we see that
    \[ h_x(u) = 
    \left(
    \frac{(r^T)^{-1}\mathbbm 1}{1-(r^{-1}x)^T\mathbbm 1}
    -
    \frac{(q^T)^{-1}\mathbbm 1}{1-(q^{-1}x)^T\mathbbm 1}
    \right)
    \left( x^T(vq^{-1})^Tx - (q^{-1}x)^T\triangle(v^Tq)\right).
    \]
Therefore \( h_x \) has rank at most 
one for all \( x \) for which it is defined
and we can conclude, by Proposition \ref{prop:PadmissibilityPropOne},
that \( S \) is \( p \)-admissible.
\end{proof}

Property A3 of Proposition \ref{prop:AdmissibilitySufficientCondition}
is a rather mysterious looking property, however we will see 
in Theorem \ref{thm:FivePointsAlwaysHaveAdmissibleMotionSpaces} below that
it has an interesting consequence for any potential proof 
of Conjecture \ref{conj:GraverTayWhiteleyConjecture}.
Now we derive a necessary condition for \( p \)-admissibility.

\begin{thm}
    \label{thm:MainTecchnicalResult}
    Let \( p \in \R^{3 \times 5} \) be a 
    \( 5 \) point configuration such that 
    \( q \) and \( r \) are invertible and 
    such that \( (r^{-1}q)^T\mathbbm1 \neq \mathbbm1 \).
    Suppose that \( S \)
    is a two dimensional 
    \( p \)-admissible subspace of \( \R^{3 \times 5} \).
    Then at least one of the following statements is true.
    \begin{enumerate}
        \item Every element of \( S \) is an affine motion
            of \( p \).
        \item \( S \) is \( p \)-equivalent to a subspace
            \( S' \) where \( S' = Ez^T \) for some 
            two dimensional subspace \( E \) of \( \R^{3} \)
            and some \( z \in \R^{5} \).
    \end{enumerate}
\end{thm}
      
Note that the requirements on \( p \) are, in particular, fulfilled 
by any generic point configuration.
The remainder of this section of the paper is devoted to proving 
Theorem \ref{thm:MainTecchnicalResult}. The strategy is as follows.
Proposition \ref{prop:PadmissibilityPropOne} allows us to 
express \( p \)-admissibility in terms of \( \mathcal P(-,q,-) \) 
and \( \mathcal P(-,r,-) \). Then a limiting argument (Lemma 
\ref{lem:PLemmaThree}) allows us to 
obtain a condition on the functions \( \mathcal L(-,q,-) \)
and \( \mathcal L(-,r,-) \) which is necessary for the 
space \( S \) to be \( p \)-admissible. 
We then express this condition in terms 
of basis elements for the space \( S \) and this allows us to reformulate
in terms of linear dependence properties of a certain
pair of polynomial functions \( \R^2 \to \R^2 \). We then
derive an elementary characterisation of 
such pairs (Lemma \ref{lem:AffinePolyProp} below).
An application of this 
result allows us to deduce the conclusion of Theorem 
\ref{thm:MainTecchnicalResult}.

\begin{lem}
    \label{lem:PLemmaThree}
    For almost all \( x \in \R^3 \)
    there is some non-zero \( u \in S \)
    such that \( \mathcal L(x,q,v)=\mathcal L(x,r,w) \)
\end{lem}

\begin{proof}
    Fix some inner product on the two dimensional real vector space
    \( S \) and let \( \tilde S \) be the unit sphere
    with respect to this inner product. 
    Let \( t \) be a positive real number. 
    By Proposition \ref{prop:PadmissibilityPropOne} 
    there is some \( u^t \in \tilde S \)
    such that \( \mathcal P(tx,q,v^t)
    = \mathcal P (tx,r, w^t)\). Since \( \tilde S \) is compact 
    there is some \( u \in \tilde S \) and some sequence 
    of real numbers \( (t_n) \) such that \( t_n \rightarrow 
    +\infty\) and \( u^{t_n} \rightarrow u  \) as \( n \rightarrow \infty
    \). Moreover, \( \mathcal P \) is a rational function and is 
    therefore continuous with respect to the standard topology
    on its natural domain. Now,
    \begin{eqnarray*}
        \mathcal L (x,q,v) &=& \lim_{n\rightarrow \infty}\frac{\mathcal P(t_nx,q,v^{t_n})}{t_n} \\
        &=& \lim_{n\rightarrow \infty}\frac{\mathcal P(t_nx,r,w^{t_n})}{t_n} \\
        &=&  \mathcal L(x,r,w)
    \end{eqnarray*}
\end{proof}

Now suppose that \( S = \langle u^1,u^2 \rangle  \) where 
\( u^i \in \R^{3 \times 5} \). 
Since \( \mathcal L (x,q,v) \) is a linear function of \( v \),
it is clear that 
there is some non-zero \( u \in S \)
such that \( \mathcal L(x,q,v)=\mathcal L(x,r,w) \)
if and only if 
\begin{equation}
    \label{eq:DependenceOne}
    c^1\mathcal L(x,q,v^1) + c^2\mathcal L(x,q,v^2)
    = c^1\mathcal L(x,r,w^1) + c^2\mathcal L(x,q,w^2)
\end{equation}
for some \( (c^1,c^2) \neq (0,0)\). Now by rearranging 
(\ref{eq:DependenceOne}) we see that
there is some non-zero \( u \in S \)
such that \( \mathcal L(x,q,v)=\mathcal L(x,r,w) \)
if and only if \( \{\{ \mathcal L(x,q,v^1)-\mathcal L(x,r,w^1),
\mathcal L(x,q,v^2)-\mathcal L(x,r,w^2)\}\}\) is a linearly
dependent multiset.

With these observations in mind, and given \( p \), \( u^1 \) and
\( u^2 \) as above, we define 
\[ f^i(x) = \mathcal L(x,q,v^i) - \mathcal L(x,r,w^i)\]
for \( i =1,2 \). So \( f^i(x) \) is defined for almost all 
\( x \in \R^{ 3} \) and \( f^i(x) \in x^\perp \). 
We have so far shown that if \( S \) is \( p \)-admissible then
\( \{\{ f^1(x),f^2(x)\}\} \) is linearly dependent for almost all \( x \). 
We now turn to an analysis of the functions \( f^i \). 

By Equation (\ref{eq:LFormula}) 
\begin{eqnarray}
    \label{eq:FAnalysisOne}
    f^i(x) &=& ((w^ir^{-1})^T-(v^iq^{-1})^T)x + \\
    &&\left(
    \frac{(q^T)^{-1}\mathbbm1 (q^{-1}x)^T(v^i)^T}{(q^{-1}x)^T\mathbbm1} - 
    \frac{(r^T)^{-1}\mathbbm1 (r^{-1}x)^T(w^i)^T}{(r^{-1}x)^T\mathbbm1} 
    \right)x \notag
\end{eqnarray}
This is a rather complicated looking expression, so the next part of the 
argument (up as far as the statement of Lemma \ref{lem:AffinePolyProp})
is devoted to transforming \( f^i \) to a more manageable form.

Suppose that \( \pi:\mathbb R^{3}\to \R^{3} \) 
is some fixed linear mapping (i.e. \( 3\times 3 \) matrix) of rank two.  
Since \( f^i(x) \in x^\perp \), we observe that 
for almost all \( x \), \(\{\{ f^1(x),f^2(x)\}\}  \) is dependent
if and only if \( \{\{ \pi(f^1(x)),\pi(f^2(x))\}\} \) is dependent.
Also, we observe that 
\begin{eqnarray*}
    q_1^T((w^ir^{-1})^T-(v^iq^{-1})^T) & = & q_1^T(w^ir^{-1})^T - (v^i_1)^T\\
    &=& 0
\end{eqnarray*}
since the first columns of \( q \) and \( r \) are identical 
and the first columns of \( v^i \) and \( w^i \) are identical.
With these observations in mind, 
we choose \( \pi \) to be the projection from \( \R^{3} \)
to \( q_1^\perp \) with kernel \( \langle (q^T)^{-1}\mathbbm1\rangle \)
and define 
\( \overline f^i = \pi\circ f^i \).
Now \( q_1^T(r^T)^{-1}\mathbbm1 = 1 \) since 
\( q_1=p_1 = r_1 \).
Therefore 
\( \pi( (r^T)^{-1}\mathbbm1) = (r^T)^{-1}\mathbbm1 
- (q^T)^{-1}\mathbbm1 \neq 0 \)
(by assumption). Let \( c = (r^T)^{-1}\mathbbm1 - (q^T)^{-1}\mathbbm1 \). 
The observations above show that 
\begin{equation}
    \overline f^i(x) = ( w^ir^{-1}-v^iq^{-1})^Tx -
    \left(\frac{c}{(r^{-1}x)^T\mathbbm1}\right)x^T(w^ir^{-1})^Tx
    \label{eq:FBarFormula}
\end{equation}
Moreover, if \( S \) is \( p \)-admissible then \( \{\{
\overline f^1(x),\overline f^2(x)\}\}\) is dependent for almost 
all \( x \).

It is clear, either from the geometric meaning of \( f^i \) 
in terms of flexes of certain frameworks, or from a direct algebraic 
calculation that replacing \( u^i \) by \( u^i+t^i\mathbbm1^T \) 
leaves the function \( f^i \) unchanged. The geometric intuition here
is that \( t^i \mathbbm1^T \) represents an infinitesimal translation 
of a particular point configuration. But infinitesimal translations
vanish in the limit that occurs in the definition of \( \mathcal L \)
and therefore contribute nothing to \( f^i \). Now we will choose 
a particular \( t^i \) as follows:
First choose some \( d \in \R^{3} \) so that \( 
\{ c,d \}\) 
is a basis for \( q_1^\perp \). 
By replacing \( u^i \) by \( u^i+t^i\mathbbm1^T \) for an appropriate
\( t^i \in \R^3 \) we can arrange that 
the column space of the matrix \( w^ir^{-1}-v^iq^{-1} \) is spanned
by \( d \). 

In other words 
\begin{equation}
    \label{eq:RankOneOne}
    (w^ir^{-1} - v^iq^{-1})^T
    = d(k^i)^T
\end{equation}
for some \( k^i \in \R^{ 3} \).
So 
\begin{equation*}
    \overline{f}^i(x)=d(k^i)^Tx-\left(\frac{c}{(r^{-1}x)^T\mathbbm1}\right)x^T(w^ir^{-1})^Tx
\end{equation*}
Now it is clear that \( \overline{f}^i \) is a rational
homogeneous function of degree 1, so 
\( \{\{{\overline f}^1(x), {\overline f}^2(x)\}\}\) 
is linearly dependent for almost all \( x \) if and only if it is 
linearly dependent for all \( x \) such that \( (r^{-1}x)^T\mathbbm1 =1 \).
Thus, define \( g^i:\{x:(r^{-1}x)^T\mathbbm1 =1\} \to (q_1)^\perp \) by 
\begin{equation}
    g^i(x) = d(k^i)^Tx - c(x^T(w^ir^{-1})^Tx).
\end{equation}
By choosing bases, we can think of \( g^i \) as 
a polynomial function from \( \R^2 \to \R^2 \)
where the first coordinate function is affine linear and the second
coordinate function is affine quadratic. Moreover, if \( S \) is 
\( p \)-admissible, then \( \{\{g^1(x),g^2(x)\}\} \) is linearly
dependent for all \( x \) such that \( xr^{-1}\mathbbm1 = 1 \).

\begin{lem}
    \label{lem:AffinePolyProp}
    For \( i=1,2 \) let \( h_i:\mathbb R^n \to \mathbb R^2 \) be of the
    form \( h_i(z) = (l_i(z),q_i(z)) \) where \( l_i:\mathbb R^n \to 
    \mathbb R\) is affine linear and \( q_i: \mathbb R^n \to \mathbb R
    \) is affine quadratic. Then 
    \( \{\{h_1(z),h_2(z)\}\} \) is linearly dependent for almost all \( z
    \in \mathbb R^n\) if and only if at least one of the following 
    conditions
    is satisfied.
    \begin{enumerate}
        \item There exist constants \( r, s \in \mathbb R \) such that \( (
            r,s) \neq (0,0)\) and
            \( rh_1 +  sh_2 \equiv 0\). (i.e. \( \{\{h_1,h_2\}\} \) 
            is linearly
            dependent over \( \mathbb R \))
        \item \( l_1 \equiv l_2 \equiv 0\) 
        \item There exists some affine linear function \( m:\mathbb R^n
            \to \mathbb R\) such that \( q_i \equiv ml_i \) for \( i=1,2 \).
    \end{enumerate}
\end{lem}

\begin{proof}
    The ``if'' direction is trivially easy to verify. 
    Now assume that \( \{\{h_1(z),h_2(z)\}\} \) is linearly dependent 
    for all \( z \in \mathbb R^n \). So 
    \(
        l_1(z)q_2(z) = l_2(z)q_1(z)
    \)
    for all \( z \in \mathbb R^n \). 
    In other words 
    \begin{equation}\label{eq:DeterminantIdentity}
        l_1q_2 \equiv l_2q_1 
    \end{equation}
    If \( l_1 \equiv 0 \) then 
    either \( l_2 \equiv 0 \) or \( q_1 \equiv 0 \). Thus, either 
    (1) or (2) is true if \( l_1 \equiv 0 \). Similarly, either (1)
    or (2) is true if \( l_2 \equiv 0 \).
    Thus from now on we assume that
    \( l_1 \not \equiv 0  \) and \( l_2 \not \equiv 0 \).
    But \( \mathbb R[X_1,\cdots,X_n] \)
    is a unique factorisation domain. 
    Therefore Equation (\ref{eq:DeterminantIdentity})
    implies that there are two possibilities. 
    Either \( l_1 \equiv \alpha l_2 \) for some  non-zero constant 
    \( \alpha \in \mathbb R \). This
    leads to case (1) of the statment. Or
    \( l_1 | q_1 \) and \( l_2 |q_2 \). 
    This leads to case (3) of the statement.
\end{proof}

Now we apply Lemma \ref{lem:AffinePolyProp} to the functions 
\( g^i \), \( i=1,2 \) and we conclude that if \( S \) is 
\( p \)-admissible then at least one of the following conditions must 
be satisfied.
\begin{enumerate}[(1')]
    \item \( \{\{g^1,g^2\}\} \) is linearly dependent over \( \R \).
    \item \( (k^1)^Tx = (k^2)^Tx = 0 \) for all \( x  \) such that 
        \( (r^{-1}x)^T\mathbbm1 = 1 \).
    \item There is some affine linear function \( m:\{x:(r^{-1}x)^T\mathbbm1
    =1\} \to \R\) such that \( x^T( w^ir^{-1})^Tx = m(x)(k^i)^Tx \)
    for all \( x \) such that \( (r^{-1}x)^T\mathbbm1 = 1 \).
\end{enumerate}

Before proceeding we remind the reader of the following elementary fact: 
Given 
\( m,m' \in \R^{3\times 3} \), then \( x^Tmx = x^Tm'x \) for almost 
all \( x \in \R^3 \) if and only if \( m-m' \) is a skew symmetric matrix.
Now suppose that (1') is true. 
This is equivalent to saying that there exists some 
\( \left( \alpha^1,\alpha^2 \right) \neq (0,0) \) such that 
\begin{equation}
    \alpha^1( w^1r^{-1} - v^1q^{-1})^T +
    \alpha^2( w^2r^{-1} - v^2q^{-1})^T
    = 0
    \label{eq:DependenceRelationOne}
\end{equation}
and 
\begin{equation}
    \alpha^1(( w^1r^{-1})^T+  w^1r^{-1})+
    \alpha^2(( w^2r^{-1})^T+  w^2r^{-1})
    = 0
    \label{eq:DependenceRelationTwo}
\end{equation}

Now Equation (\ref{eq:DependenceRelationTwo}) 
implies that \( (\alpha^1  w^1
+\alpha^2  w^2)r^{-1}\) is skew symmetric. Transposing
Equation 
(\ref{eq:DependenceRelationOne}) and rearranging, we see that 
\( (\alpha^1  v^1
+\alpha^2  v^2)q^{-1} =(\alpha^1  w^1
+\alpha^2  w^2)r^{-1}  \).
Thus we see that if (1') is true then \( \alpha^1 u^1
+\alpha^2 u^2\) is an non-zero infinitesimal isometry of \( p \)
which contradicts our assumption that \( S\cap \mathcal I_p = 0 \). 
Under the assumptions of Theorem \ref{thm:MainTecchnicalResult}
we conclude that (1') cannot be true.

Suppose that (2') is true. So \( ( w^ir^{-1})^T -
( v^iq^{-1})^T = 0\) for \( i=1,2 \). 
Therefore, by Lemma \ref{lem:LinearMotionLemmaOne}, 
\( {u}^i \) is a linear motion of \( p \)  for \( i=1,2 \). It follows
easily from this that every element of \( S \) is a linear 
motion of \( p \). Remembering that we replaced \( u^i \) by \( u^i+t^i\mathbbm1^T \)
earlier in the argument, we conclude that every element of our original
\( S \) is in fact an affine motion of \( p \).

Finally, suppose that (3') is true. Then 
\( m(x) = x^Tl \)
for some \( l\in \R^{3} \). Therefore 
\( x^T( w^ir^{-1})^Tx = x^Tl(k^i)^Tx \)
for almost all \( x  \).
As remarked above this implies that 
\( ( w^ir^{-1})^T  = l(k^i)^T +a^i\)
for some skew symmetric \( a^i \in \R^{3 \times 3} \). 
Combining this this with Equation (\ref{eq:RankOneOne}),
we see that \( ( v^iq^{-1})^T = (l-d)(k^i)^T+a^i \). 
Therefore we have
\begin{eqnarray}
     v^i+a^iq &=& k^i(l-d)^Tq \label{eq:RankOneTwo}\\
     w^i+a^ir &=& k^il^Tr \label{eq:RankOneThree}
\end{eqnarray}
Now let \( z \in \R^{ 5}\) be the unique vector such that
deleting the second and third rows yields \( q^T(l-d) \),
whereas deleting the fourth and fifth rows of \( z \)
yields \( r^Tl \) (this makes sense since 
\( q_1^Td = 0 \) and \( r_1 = q_1 \)). 
Let \( E \) be the subspace of \( \R^{3}\)
that is spanned by \( \{k^1,k^2\} \). Then
Lemma \ref{lem:pEquivalenceLemma} and Equations 
(\ref{eq:RankOneTwo}) and (\ref{eq:RankOneThree}) imply 
that \( S \) is \( p \)-equivalent to \( Ez^T \).
This completes the proof of Theorem \ref{thm:MainTecchnicalResult}.
\qed

\section{Applications to Rigidity Theory}
\label{sec:Applications}

The remainder of the paper will be concerned with applications
of the results of the sections \ref{sec:PAdmissible},
\ref{sec:KNoneFrameworks} and \ref{sec:FivePoints} to 
the rigidity theory of frameworks.

\begin{lem}
    \label{lem:ThreePointsFixed}
    Let \( p \in \R^{3\times 5} \) be generic and 
    let \( S \) be a two dimensional \( p \)-admissible 
    subspace of \( \R^{3\times 5} \). Suppose that
    there is some three point subconfiguration of \( p \)
    such that  that for every \( u \in S \),
    \( u \) restricts to a space of isometries
    of that three point configuration.
    Then, up to a possible permutation of the points, \( S \) is \( p \)-equivalent 
    to one of the spaces described in Example \ref{exa:AdmissibleExampleOne}
    or Example \ref{exa:AdmissibleExampleTwo}.
\end{lem}

\begin{proof}
    Using the notational convention of Section \ref{sec:FivePoints},
    we may assume without loss of generality that \( w = 0 \)
    for all \( u \in S \). So \( \mathcal L(x,r,w) = 0  \) for all
    \( x \). Therefore by Lemma \ref{lem:PLemmaThree} the linear mapping
    \( u \mapsto \mathcal L(x,q,v) \) has rank at most one for almost all
    \( x \). Now \( \mathcal L (x,q,v) = -(q^T)^{-1}\left( I
    -\frac{\mathbbm1 (q^{-1}x)^T}{(q^{-1}x)^T\mathbbm1}\right)v^Tx \).
    But the projection matrix \(   I -\frac{\mathbbm1 
    (q^{-1}x)^T}{(q^{-1}x)^T\mathbbm1}\) has kernel \( \langle
    \mathbbm1 \rangle\) and the first row of \( v^T \) is zero. 
    Therefore we conclude that
    the linear mapping \( u\mapsto v^Tx \) has rank at most one for almost 
    all \( x \). It follows easily that \( S = Ez^T \) where \( E \) 
    is a two dimensional subspace of \( \R^3 \) and \( z^T =
    (0,0,0,k,1)\) for some constant \( k \) (here we may 
    have to permute vertices 4 and 5 to ensure that the fifth entry 
    of \( z^T \) is non zero). If \( k =0 \) then \( S \)
    is \( p \)-equivalent to the space described in Example \ref{exa:AdmissibleExampleOne}.
    It remains to show that if \( k \neq 0 \) then we must have \( E = (p_5-p_4)^\perp \).
    Now, since \( S = Ez^\perp \) is \( p \)-admissible and since \( \mathcal P(x,r,w) =0\)
    for almost all \( x \) and all \( u \in S \),
    it follows from Definition \ref{defn:Admissibility} that 
    for almost all \( x \), there is some 
    non-zero \( \nu^x \in E \) such that 
    \begin{eqnarray}
        ( \nu^x)^T(x-p_4)&=&0 \label{eq:NuOne}\\
        ( \nu^x)^T(x-p_5)&=&0\label{eq:NuTwo}
    \end{eqnarray}
    Subtracting (\ref{eq:NuTwo}) from (\ref{eq:NuOne}) yields \( \nu^x 
    \in E \cap (p_5-p_4)^\perp  \).
    If this space has dimension at most 1, then it clearly follows that \( \nu^x = 0 \)
    for almost all \( x \) which contradicts the fact that 
    \( \nu^x \neq 0 \) for almost all \( x \). 
    Therefore \( \dim(E \cap (p_5-p_4)^\perp) = 2 \)
    as required.
\end{proof}
We can use this to prove a special case of Conjecture \ref{conj:GraverTayWhiteleyConjecture}.
We introduce the following notation for the rest of the section. 
\( G=(V,E) \) is a generically 3-isostatic graph,
\( X =\{1,2,3,4,5\}\subset V \), \( \{e,f\} \subset E(X) \)
and \( G' \) is the 2-extension of \( G \) supported on vertex set \( X \)
and edge set \( \{e,f\} \).
Moreover, given a framework \( (G,\rho) \), let \( p \in \R^{3\times 5} \)
be defined by \( p_i = \rho_i \) for \( i \in X \).

\begin{cor}
    \label{cor:FourEdges}
    Suppose that \( G-\{e,f\} \) contains an implied 
    subgraph \( H \) that is isomorphic to the graph shown below
            \begin{center}
\definecolor{qqqqff}{rgb}{0,0,1}
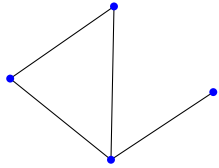
\begin{tikzpicture}[line cap=round,line join=round,>=triangle 45,x=1.0cm,y=1.0cm]
\clip(-4.3,2.82) rectangle (0.82,6.3);
\draw (-3.12,4.5)-- (-1.74,5.46);
\draw (-3.12,4.5)-- (-1.78,3.42);
\draw (-1.74,5.46)-- (-1.78,3.42);
\draw (-1.78,3.42)-- (-0.42,4.32);
\begin{scriptsize}
    \fill [color = qqqqff] (-3.12,4.5) circle (1.5pt);
    \fill [color = qqqqff] (-1.74,5.46) circle (1.5pt);
    \fill [color = qqqqff] (-1.78,3.42) circle (1.5pt);
    \fill [color = qqqqff] (-0.42,4.32) circle (1.5pt);
\end{scriptsize}
\end{tikzpicture}
\end{center}
    and where the vertices of \( H \) are in \( X \).
    Then either \( G-\{e,f\} \) contains an implied \( K_4 \)
    whose vertices are in \( X \),
    or \( G' \) is generically 3-rigid.
\end{cor}

\begin{proof}
    Let \( (G,\rho) \) be a generic framework.
    Suppose that \( G' \) is not generically rigid.
    Then the nontrivial flexes of \( (G-\{e,f\},\rho) \) induce 
    a \( p \)-admissible subspace of \( \R^{3\times5} \). Since 
    \( H \) contains a triangle, we can apply Lemma 
    \ref{lem:ThreePointsFixed} to conclude that this space is 
    \( p \)-equivalent to either Example 
    \ref{exa:AdmissibleExampleOne} or Example 
    \ref{exa:AdmissibleExampleTwo}. 
    In the first case, it is clear that \( G-\{e,f\} \)
    contains an implied \( K_4 \).
    So it suffices to show that the space from Example 
    \ref{exa:AdmissibleExampleTwo} cannot arise from this 
    framework. However this clearly follows from the genericity
    of \( p \) and the constraint imposed by the fourth 
    edge of \( H \).
\end{proof}

Our next result illustrates a fundamental difficulty that must be resolved directly or indirectly by any possible proof of Conjecture \ref{conj:GraverTayWhiteleyConjecture}.

\begin{thm}
    \label{thm:FivePointsAlwaysHaveAdmissibleMotionSpaces}
    For any generic \( p \in \R^{3 \times 5} \) there is a 
    two dimensional \( p \)-admissible subspace
    \( S \) of \(\R^{3 \times 5} \) such that
    for any 3 point subconfiguration \( p' \) of \( p \),
    there is some \( u\in S \) that is not an isometry of \( p' \).
\end{thm}

\begin{proof}
    In fact we will show that there are infinitely many such subspaces.
    The proof essentially amounts to a dimension count based on 
    Proposition \ref{prop:AdmissibilitySufficientCondition}. 
    Let \( L \) be the linear subspace of \( \R^{3\times 5} \) 
    consisting of the linear motions of \( p \). Clearly,
    \( L \) has dimension 9 - it is isomorphic to \( \R^{3\times 3} \). 
    Now consider the equation 
    \begin{equation}
        (q^T)^{-1}\triangle(v^Tq) = (r^T)^{-1}\triangle(w^Tr)
        \label{eq:StressConstraint}
    \end{equation}
    as linear constraint on \( L \). This constraint has rank 
    at most two, since \( q_1 = r_1  \) and \( v_1 = w_1 \). 
    Let \( R \) be the subspace of \( L \) consisting of solutions
    of Equation (\ref{eq:StressConstraint}). The dimension of \( R \)
    is at least 7 and \( R \cap \mathcal I_p \) is a three 
    dimensional linear
    subspace of \( R \). 
    Proposition \ref{prop:AdmissibilitySufficientCondition}
    implies that any two dimensional subspace of \( R \) that intersects
    \( \mathcal I_p \) trivially is in fact \( p \)-admissible.
    We also observe that two dimensional linear subspaces 
    \( S \) and \( S' \)
    of \( R \) are \( p \)-equivalent if and only 
    they project to the same subspace of \( R/(R\cap \mathcal I_p) \).
    Therefore there is a one to one correspondence between 
    the set of \( p \)-equivalence classes of two dimensional \( p \)-admissible 
    subspaces of \( R \) and the set of two dimensional linear subspaces 
    of \( R/(R\cap \mathcal I_p) \). 
    This latter set forms 
    a Grassman manifold \( M \) of dimension at least 4, since the 
    dimension of \( R/(R\cap \mathcal I_p) \) is at least 4.

    Now suppose that \( S \) is 
    one of the \( p \)-admissible spaces constructed in Examples 
    \ref{exa:AdmissibleExampleOne} or \ref{exa:AdmissibleExampleTwo}.
    It is easy to see that for a given \( p \) 
    there are only finitely many such spaces that consist of affine motions.
    We also note that if \( S \) and \( S' \) are \( p \)-equivalent 
    spaces and \( S \) consists of affine motions, then \( S' \) consists
    of affine motions. This follows from the fact that every trivial 
    motion of \( p \) is also an affine motion.
    Therefore there must be many elements of \( M \) 
    that are not \( p \)-equivalent
    to either Example \ref{exa:AdmissibleExampleOne} or Example \ref{exa:AdmissibleExampleTwo}. 
    By Lemma \ref{lem:ThreePointsFixed} this means that there are many
    elements of \( M \) that do not restrict to spaces of isometries for 
    any three point subset of \( p \).
\end{proof}

On the other hand we have the following proposition concerning 
one dimensional \(p  \)-admissible spaces, where 
\( p \in \R^{n \times (n+1)}\) is in general position.

\begin{prop}
    \label{prop:OneDimAdmissible}
    Let \( p \in \R^{ n \times (n+1) }\) be in general 
    position. There is no one dimensional \( p \)-admissible 
    subspace of \( \R^{n\times(n+1)} \).
\end{prop}
    
\begin{proof}    
    Let \( q \), respectively \( r \), be obtained 
    by deleting the \( n \)th, respectively the \( (n+1) \)st,
    column of \( p \). Replacing \( p \) by a 
    suitable translate if necessary, we may assume 
    that \( q \) and \( r \) are both invertible. 
    Now our hypothesis ensures that
    \( (q^T)^{-1}\mathbbm 1 \neq (r^T)^{-1}\mathbbm 1 \).
    Moreover, since \( r^{-1}q  \) is identical 
    to \( I \) in its first two columns, it is clear that
    \( \{\{(q^T)^{-1}\mathbbm1, (r^T)^{-1}\mathbbm1 \}\} \) is in 
    fact linearly independent.
    Suppose that \( S = \langle u\rangle \) is 
    a one dimensional \( p \)-admissible subspace 
    of \( \R^{n \times (n+1)} \).
    Let \( v \), respectively \( w \), be obtained
    by deleting the \( n \)th, respectively the 
    \( (n+1) \)st, columns of \( u \). Then 
    \( \mathcal L(x,q,v) - \mathcal L(x,r,w) = 0 \)
    for all \( x \). 
    Therefore \[ -(vq^{-1})^Tx + (wr^{-1})^Tx
        +(q^T)^{-1}\mathbbm1 \frac{x^T(vq^{-1})^Tx}{(q^{-1}x)^T\mathbbm 1}
        -(r^T)^{-1}\mathbbm1 \frac{x^T(wr^{-1})^Tx}{(r^{-1}x)^T\mathbbm 1}
        =0\]
        for all \( x \). But \( \{\{(q^T)^{-1}\mathbbm1,(r^T)^{-1}
    \mathbbm 1\}\}\) is linearly independent. Therefore 
    the rational functions \( \frac{x^T(vq^{-1})^Tx}{(q^{-1}x)^T
    \mathbbm1}\) and \( \frac{x^T(wr^{-1})^Tx}{(r^{-1}x)^T\mathbbm1} \)
    must both in fact be linear functions. In other words,
    \( x^T(vq^{-1})^Tx = x^T(q^{-1})^T\mathbbm 1 z^Tx \) for all \( x \)
    and for some \( z \in \R^{3} \). It follows that
    \( v = z\mathbbm 1^T +aq \) for some skew symmetric matrix 
    \( a \). In other words, \( v \) is 
    an isometry of \( q \). Similarly \( w \) 
    an isometry of \( r \). 
    Now since
    \( \mathcal P(x,q,v) - \mathcal P(x,r,w) = 0 \)
    for all \( x \), it follows easily that 
    \( u \) is an isometry of \( p \),
    contradicting the fact that \( S\cap \mathcal I_p = 0 \).
\end{proof}

Proposition \ref{prop:OneDimAdmissible} immediately implies
the well known fact (see, for example Chapter 5 of 
\cite{MR1251062})
that any Henneberg \( 1 \)-extension of a generically \( n \)-isostatic
graph is also \( n \)-isostatic.
Indeed most proofs of that fact in the literature essentially
amount to proofs of Proposition \ref{prop:OneDimAdmissible}.
We have included another proof here for the sake of completeness
and also to illustrate that our matrix algebraic viewpoint
provides another way to understand some well known geometric
results from the literature.

On the other hand any proof of 
Conjecture \ref{conj:GraverTayWhiteleyConjecture} 
must somehow (directly or indirectly) demonstrate that the motions
whose existence is asserted by Theorem 
\ref{thm:FivePointsAlwaysHaveAdmissibleMotionSpaces} 
cannot arise by deleting two edges from a generic framework. 
Thus Theorem \ref{thm:FivePointsAlwaysHaveAdmissibleMotionSpaces}
identifies a fundamental difficulty with Conjecture \ref{conj:GraverTayWhiteleyConjecture}
that does not arise for the corresponding 
question about Henneberg 1-extensions.

\begin{lem}
    \label{lem:ConicAtInfinity}
    Suppose that \( S \) is a two dimensional space 
    of affine motions of a generic \( p \in \R^{3 \times 5} \) and
    suppose that \( S \) restricts to a space of isometries on 
    each of five distinct 
    edges of \( K_5 \). Then there is some non-zero \( u \in S \) that
    is an infinitesimal isometry of \( p \).
\end{lem}

\begin{proof} 
    The proof is a variation of the standard 
    `conic at infinity' argument (see
    Proposition 4.2 of \cite{MR2132290}, for example).
    First we observe that there are 6 different isomorphism
    types of 5 edge graphs on 5 vertices. One readily checks that
    for each of these 6 graphs it is possible to add one edge
    so that the resulting graph is isomorphic to one of the three 
    graphs shown in Figure \ref{fig:SixEdgeGraphs}.
    Now since \( S \) is
    two dimensional it is clear that for any edge \( e \) of \( K_5 \) 
    there is some non-zero \( u \in S \) such that \( u \)
    is an isometry on \( e \). Thus we we can find some non-zero
    \( u \in S \) such that \( u \) is a flex of an \( H \)-framework,
    where \( H \)
    is isomorphic to one of the three graphs shown in Figure 
    \ref{fig:SixEdgeGraphs}.
    \begin{figure}
\definecolor{qqqqff}{rgb}{0,0,1}
\begin{tikzpicture}[line cap=round,line join=round,>=triangle 45,x=1.0cm,y=1.0cm]
\clip(-4.3,2.62) rectangle (7.42,6.3);
\draw (-2.6,5.62)-- (-3.16,4.48);
\draw (-3.16,4.48)-- (-2.16,4.48);
\draw (-2.16,4.48)-- (-2.6,5.62);
\draw (-3.16,4.48)-- (-2.64,3.4);
\draw (-2.64,3.4)-- (-2.16,4.48);
\draw (-2.6,5.62)-- (-1.48,5.48);
\draw (1.1,5.54)-- (0.52,4.36);
\draw (0.52,4.36)-- (1.52,4.36);
\draw (1.52,4.36)-- (1.1,5.54);
\draw (0.52,4.36)-- (1.02,3.32);
\draw (1.02,3.32)-- (1.52,4.36);
\draw (1.52,4.36)-- (2.18,5.08);
\draw (4.24,4.58)-- (5.22,5.46);
\draw (5.22,5.46)-- (6.06,4.64);
\draw (6.06,4.64)-- (4.24,4.58);
\draw (4.24,4.58)-- (4.68,3.46);
\draw (4.68,3.46)-- (5.74,3.5);
\draw (5.74,3.5)-- (6.06,4.64);
\begin{scriptsize}
\fill [color=qqqqff] (-2.6,5.62) circle (1.5pt);
\fill [color=qqqqff] (-3.16,4.48) circle (1.5pt);
\fill [color=qqqqff] (-2.16,4.48) circle (1.5pt);
\fill [color=qqqqff] (-2.64,3.4) circle (1.5pt);
\fill [color=qqqqff] (-1.48,5.48) circle (1.5pt);
\fill [color=qqqqff] (1.1,5.54) circle (1.5pt);
\fill [color=qqqqff] (0.52,4.36) circle (1.5pt);
\fill [color=qqqqff] (1.52,4.36) circle (1.5pt);
\fill [color=qqqqff] (1.02,3.32) circle (1.5pt);
\fill [color=qqqqff] (2.18,5.08) circle (1.5pt);
\fill [color=qqqqff] (4.24,4.58) circle (1.5pt);
\fill [color=qqqqff] (5.22,5.46) circle (1.5pt);
\fill [color=qqqqff] (6.06,4.64) circle (1.5pt);
\fill [color=qqqqff] (4.68,3.46) circle (1.5pt);
\fill [color=qqqqff] (5.74,3.5) circle (1.5pt);
\end{scriptsize}
\end{tikzpicture}
\begin{caption}{The three possibilities for the graph
        \( H \) from the proof of Lemma \ref{lem:ConicAtInfinity}}
        \label{fig:SixEdgeGraphs}
    \end{caption}
\end{figure}

    Since \( u \) is affine, we have \( u_i = mp_i +b \) for 
    some \( m \in \R^{3 \times 3} \) and \( b \in \R^3 \). Therefore
    \begin{equation}
        \label{eq:ConicAtInf}
        (p_i-p_j)^Tm(p_i-p_j) = 0
    \end{equation}
    for all \( ij \) in \( H \). In other words, the 6 edges of a 
    generic \( H \)- framework lie on the conic associated to the 
    matrix \( m \).
    Now we show that the six edges of a generic \( H \)-framework
    cannot lie on a non-trivial conic at infinity. This follows from
    the observation that since \( H \) contains a triangle of edges, any
    such conic must be a degenerate one (i.e. a pair of lines). Thus 
    in the two cases where \( H \) has a pair of triangles, the remaining
    edge will not (generically) lie on the conic spanned by the pair of lines.
    In the other case \( H \) consists of a triangle and three edges that 
    do not form a triangle. These latter three edges will not generically 
    lie on a single line, nor will any of them lie on the line correponding to
    the triangle.
    Therefore Equation (\ref{eq:ConicAtInf}) must in fact
    be a trivial quadratic equation. 
    In other words, \( m \) is skew symmetric 
    and it follows that \( u \) is 
    a trivial motion of \( p \).
\end{proof}

\begin{thm}
    \label{thm:FiveEdgesLeft}
    Suppose that \( |E(X)| \geq 7  \). Then 
    either \( G -\{e,f\}\) contains an implied \( K_4 \)
    whose vertices are in \( X \),
    or \( G' \) is generically 3-rigid.
\end{thm}

\begin{proof}
    Suppose that \( G' \) is not generically 3-rigid and let 
    \( (G,\rho) \) be a generic framework. 
    So \( (G-\{e,f\},\rho) \) has a two dimensional space of 
    flexes that induces a \( p \)-admissible subspace, \( S \), of 
    \( \R^{3\times5} \). Suppose that \( S \) consists of 
    affine motions of \( p \). Since \( \dim(S)=2 \), 
    there is some non-zero \( u \in S \)
    that is nontrivial on \( p \) that restricts to an isometry 
    of 6 edges of \( K(X) \). 
    By Lemma \ref{lem:ConicAtInfinity} 
    some non-zero element of \( S \) is 
    a trivial motion of \( p \) which contradicts the 
    fact that \( S \) is \( p \)-admissible. 
    Therefore by Theorem \ref{thm:MainTecchnicalResult},
    we can assume 
    that \( S= Ez^T \) where \( E \leq \R^3 \) has dimension two 
    and \( z \in \R^5 \).
    Now suppose that \( z_1=z_2=z_3 \). Then \( S \) restricts to 
    a space of isometries of \( \{p_1,p_2,p_3\} \). 
    But \( S \) restricts to a space of 
    isometries of at least two edges of \( E(X)-E(\{1,2,3\}) \). Therefore, 
    by Corollary \ref{cor:FourEdges} we conclude that either
    \( G-\{e,f\} \) contains an implied \( K_4 \) with vertices
    in \( X \) or \( G' \) is generically rigid. 
    So, from now on we can assume that
    in the multiset \( \{\{z_1,z_2,z_3,z_4,z_5\}\} \),
    no value occurs with multiplicity greater than two.
    We will show that this assumption leads to a contradiction.
    Up to a permutation, there are three cases to consider:
    \begin{enumerate}
        \item \( z_1,z_2,z_3 \) are pairwise distinct,
            \( z_4= z_1\) and \( z_5=z_2 \).
        \item \( z_1,z_2,z_3,z_4 \) are pairwise distinct and
            \( z_5 = z_1 \).
        \item \( z_1,z_2,z_3,z_4,z_5 \) are pairwise distinct.
    \end{enumerate}
    We will give details only for first of these cases. 
    So suppose that \( z_1,z_2,z_3 \) are pairwise distinct, 
    and that \( z_4= z_1\) and \( z_5=z_2 \).
    Since \( \dim(S) = 2 \) we can find a nonzero 
    element \( cz^T \in S \), \( c \in E \),
    that is an isomtery of six edges of \( K(X) \).
    But \( cz^T \) is an isometry of the edge \( ij \)
    if and only if \( (z_i-z_j)c^T(p_i-p_j)=0 \). 
    Now, under our working assumption,
    \( z_i-z_j = 0 \) if and only if \( ij = 14 \) or 
    \( ij = 25 \). Therefore there
    are at least four of the six edges for which \( z_i-z_i\neq 0 \). 
    So there are at least four edges for which \( c^T(p_i-p_j) =0 \). 
    But since \( p \) is generic, it follows easily that \( c = 0 \)
    which contradicts the fact that \( cz^T \neq 0 \). 
    The other two cases can be dealt with 
    in a similar way. 
\end{proof}

Theorem \ref{thm:FiveEdgesLeft}
generalises the result of Graver mentioned in the 
introduction. In his result, the five edges in \( E(X) -\{e,f\}
\) are required to form a 5-cycle. 

\section{Acknowledgements}

The author would like to thank Bob Connelly for some very helpful
conversations on this topic
and in particular for pointing out the 
application of the conic at infinity argument 
in the proof of Theorem \ref{thm:FiveEdgesLeft}.
The author also expresses his thanks to Bill Jackson and John Owen 
for allowing him access to their unpublished manuscript \cite{JacksonOwenNotes}.

Finally I would like to thank my daughter Lily for keeping me awake.

\end{document}